\documentclass[12pt]{amsart}
\usepackage{amsfonts}
\usepackage{amsfonts,latexsym,rawfonts,amsmath,amssymb,amsthm}
\usepackage[plainpages=false]{hyperref}

\usepackage{graphicx}

\RequirePackage{color}

\numberwithin{equation}{section}

\newtheorem{prop}{Proposition}[section]
\newtheorem{thm}[prop]{Theorem}
\newtheorem{lem}[prop]{Lemma}

\newtheorem{rem}[prop]{Remark}

\allowdisplaybreaks
\title[Liouville Theorems for a class of Monge-Amp\`ere equations]{Liouville Theorems for a class of degenerate or singular Monge-Amp\`ere equations}
\author[L.~Wang \& B.~Zhou]{Ling Wang and Bin  Zhou}

\address{School of Mathematical Sciences, Peking
University, Beijing 100871, China.}
\email{lingwang@stu.pku.edu.cn}

\address{School of Mathematical Sciences, Peking
University, Beijing 100871, China.}
\email{bzhou@pku.edu.cn}

\thanks {This research is partially supported by  NSFC  grants 12271008 and National Key R$\&$D Program of China SQ2020YFA0712800.}

\begin{document}
\subjclass[2020]{35J96, 35J70, 35B53, 35A09}
\keywords{degenerate Monge-Amp\`ere equation, Liouville theorems, partial Legendre transform, method of moving spheres}

\maketitle

\begin{abstract}
	In this note, we classify solutions to a class of Monge-Amp\`ere equations whose right hand side may be degenerate or singular in the half space. 
	Solutions to these equations are special solutions to a class of fourth order equations, including the affine maximal hypersurface equation, in the half space.
Both the Dirichlet boundary value and Neumann boundary value cases are considered. 
		
\end{abstract}

\section{Introduction}

The main purpose of this paper is to investigate Liouville theorems for the following class of Monge-Amp\`ere equations
\begin{equation}\label{eq:MA-d-hf-n}
\det D^2u=(a+bx_n)^\alpha,\ \ \alpha\in \mathbb R
\end{equation} 
in the half space $\mathbb R^n_+:=\left\{(x',x_n)\in\mathbb R^n:x_n>0\right\}$, where $a\geq 0$ and $b>0$.  
A motivation to consider \eqref{eq:MA-d-hf-n} comes from the study of the
following class of fourth order equations
\begin{equation}\label{eq:4}
U^{ij} w_{ij}=0, 
\end{equation}
where $(U^{i j})$ denotes the cofactor matrix of $(u_{i j})$ and $w=(\det D^{2} u)^{-\theta}$, $\theta\in \mathbb R$ ($\theta\neq 0$). In particular, when $\theta=\frac{n+1}{n+2}$, it is the \textit{affine mean curvature equation} in affine geometry \cite{Ch} and when $\theta=1$, it is \textit{Abreu's equation} \cite{Ab}. A first breakthrough for the study of this class of equations is the Chern conjecture, also known as the {\it affine Bernstein theorem} now, solved by Trudinger-Wang \cite{TW1}, which says an entire strictly, uniformly convex solution to \eqref{eq:4} on $\mathbb R^2$ when $\theta=\frac{3}{4}$ must be a quadratic polynomial. Later, it is shown that the Bernstein theorem also holds when $\frac{3}{4}<\theta\leq 1$ \cite{JL, Z} and $\theta<0$ \cite{TW2}. When we consider \eqref{eq:4} in $\mathbb{R}_{+}^{n}$, one can easily find solutions where are not quadratic polynomials. In particular, solutions to \eqref{eq:MA-d-hf-n} with $\alpha=-\frac{1}{\theta}$
give a class of special solutions to \eqref{eq:4}.  Note that in the entire space case, \eqref{eq:MA-d-hf-n} admits no
convex solutions if $b\neq 0$.


When $a=0$, \eqref{eq:MA-d-hf-n} may be degenerate ($\alpha>0$) or singular ($\alpha<0$) on $\partial\mathbb R^n_+$. When $\alpha\geq 0$, 
Savin  \cite{S1, S2} proved that if the Dirichlet boundary value $u(x',0)=\frac{1}{2}|x'|^2$ is assigned, any convex continuous solution to  \eqref{eq:MA-d-hf-n} with the growth condition $u=O(|x|^{3+\alpha-\varepsilon})$ as $|x|\to+\infty$ must be the form of
$$u(Ax)=Bx_n+\frac{1}{2}|x'|^2+\frac{x_n^{2+\alpha}}{(2+\alpha)(1+\alpha)}$$
for some sliding $A$ along $x_n = 0$, and some constant $B$. In particular, when $\alpha=0$, the solution is a quadratic polynomial. 
This result was later extended to the singular case with $\alpha\in(-1,0)$ by Savin and Zhang \cite{SZ}. There are examples show that the growth condition at infinity is necessary in general dimensions. When $\alpha=-1$, the local  asymptotic behavior of the solution near the boundary in dimension  two was studied in \cite{Ru}.

In this paper, we concentrate on the two dimensional case.
Our first result classifies all solutions to \eqref{eq:MA-d-hf-n} with Dirichlet condition in dimension two 
 when $\alpha>-2$. 
\begin{thm}\label{thm:MA-d-hf}
Let $u(x,y)\in C^2({\mathbb R_+^2})\cap C(\overline{\mathbb R_+^2})$ be a convex solution to
\begin{equation}\label{eq:MA-d-hf}
\left\{
\begin{aligned}
\det D^2u&=(a+by)^\alpha &&\text{in }\mathbb R^2_+,\\[5pt]
u(x,0)&=\frac{1}{2}x^2 &&\text{on }\partial\mathbb R^2_+ ,
\end{aligned}
\right.
\end{equation}
where $a\geq 0$, $b>0$, and $\alpha>-2$. Then there exist $A,\,B,\,C\in\mathbb R$ with $A\geq 0$ such that
\begin{equation}\label{eq:s-MA}
u(x,y)=\left\{
\begin{aligned}
&\frac{(b-aA)(a+by)^{2+\alpha}}{b^3(1+\alpha)(2+\alpha)}+\frac{A(a+by)^{3+\alpha}}{b^3(2+\alpha)(3+\alpha)}-By\\
&\quad-\frac{(b-aA)a^{2+\alpha}}{b^3(1+\alpha)(2+\alpha)}-\frac{Aa^{3+\alpha}}{b^3(2+\alpha)(3+\alpha)}+\frac{(x-Cy)^2}{2(1+Ay)},&&\alpha\neq -1;\\[7pt]
&\frac{b-aA}{b^3}(a+by)\ln(a+by)+\frac{A}{2b}y^2-By\\
&\quad\quad\quad\quad\quad\quad\quad\quad\quad\,\,-\frac{(b-aA)a\ln a}{b^3}+\frac{(x-Cy)^2}{2(1+Ay)},&&\alpha= -1.
\end{aligned}
\right.
\end{equation}
\end{thm}

\begin{rem}\quad
\begin{enumerate}
\item  When $a=0$, we improve the exponent in the results of \cite{S1, S2, SZ} to $\alpha>-2$ in two dimensional case. This exponent 
is sharp since \eqref{eq:MA-d-hf} admits no solutions continuous up to the boundary for $\alpha\leq -2$ (see details in Remark \ref{rem:sharp}).
When $\alpha=0$, Theorem \ref{thm:MA-d-hf} can be also found in \cite[Page 145-148]{Fi}.

\item If we assume $u=O(|(x,y)|^{3+\alpha-\varepsilon})$ as $|(x,y)|\to+\infty$, then we have that  $A$ must be $0$ in \eqref{eq:s-MA}. Hence we can recover some of the results in \cite{S2,SZ} in dimension two.

\item In a subsequent work, we are going to study the Liouville type theorem the following problem
\begin{equation}\label{eq:4th-eq-hf}
\left\{\begin{aligned}
U^{i j} w_{i j}&=0  &&\text { in } \mathbb{R}_{+}^{n}, \\
u&=\frac{1}{2}|x'|^2  &&\text { on } \partial \mathbb{R}_{+}^{n},\\
w&=1  &&\text { on } \partial \mathbb{R}_{+}^{n}.
\end{aligned}\right.
\end{equation}

\end{enumerate}
\end{rem}

The main idea to prove Theorem \ref{thm:MA-d-hf} is as follows. Let  $u(x, y)$ be a uniformly convex solution to \eqref{eq:MA-d-hf-n}.
Then its partial Legendre transform in the $x$-variable is
\begin{equation}
u^{\star}(\xi, \eta)=x u_{x}(x, y)-u(x, y),
\end{equation}
where $(\xi, \eta)=(u_{x}, y)$. It is easy to check that $u^\star$ is a solution to
\begin{equation}\label{eq:gru}
(a+b\eta)^{\alpha}u_{\xi \xi}^{\star}+u_{\eta\eta}^{\star}=0.
\end{equation}
When $a=0$, this Grushin type equation was studied in \cite{CS}. By a change of variables
$v(x_1,x_2)=u^\star\left(x_1,f(x_2)\right)$,
where
\[\xi=x_1, \ \eta=f(x_2)=b^{\frac{-\alpha}{\alpha+2}}\left(\frac{\alpha+2}{2}x_2\right)^{\frac{2}{\alpha+2}}-\frac{a}{b},\]
we know that $v$ solves
the following divergence type equation
\begin{equation}\label{linear-div}
\operatorname{div}\left(x_2^{\frac{\alpha}{\alpha+2}}\nabla v\right)=0,
\end{equation}
 which may be degenerate or singular. A Liouville theorem for \eqref{linear-div} on the upper half space has been obtained recently by \cite{WZ}. However, 
 in our case, the domain may shift after the transformations. Hence, we need a generalization of the result in \cite{WZ} to general upper half spaces.

The above approach also works for the case of Neumann problem. As for the Neumann boundary value case, we only consider the degenerate case, stated as follows.

\begin{thm}\label{thm:Liou-d-hf-ne}
Let $u(x,y)\in C^2({\mathbb R_+^2})\cap C^1(\overline{\mathbb R_+^2})$ be a convex solution to
\begin{equation}\label{eq:MA-d-hf-ne}
\left\{
\begin{aligned}
\det D^2u&=y^\alpha &&\text{in }\mathbb R^2_+,\\
u_y(x,0)&=0 &&\text{on }\partial\mathbb R^2_+ ,
\end{aligned}
\right.
\end{equation}
where $\alpha\geq 0$. Then there exist a constant $A>0$, and a linear function $l(x)$ such that
\begin{equation}\label{eq:s-MA-d-hf-ne}
u(x,y)=\frac{1}{2A}x^2+\frac{A}{(2+\alpha)(1+\alpha)}y^{2+\alpha}+l(x).
\end{equation}
\end{thm}
\begin{rem}\quad
\begin{enumerate}
\item When $\alpha=0$, Theorem \ref{thm:Liou-d-hf-ne} is included in \cite[Theorem 1.1]{JT}. In fact, it is proved in \cite{JT} that any convex solution to Neumann problem of Monge-Amp\`ere equations in the half plane must be a quadratic polynomial for two dimensional case, and the conclusion still holds for dimension $n \geq 3$ if either the boundary value is zero or the solution restricted on some $n-2$ dimensional subspace is bounded from above by a quadratic function. Here we extend this to the degenerate case.

\item It is still unknown whether Theorem \ref{thm:Liou-d-hf-ne} is true for $\alpha<0$. Although our method doesn't work for this case, we believe the conclusion is true for $\alpha>-1$.
\end{enumerate}
\end{rem}

Finally, we turn to the Liouville theorem on the whole space.
The celebrated result of J\"orgens \cite{Jo}, Calabi \cite{Ca} and Pogorelov \cite{Po} states that any entire classical convex solution to the Monge-Amp\`ere equation 
$$\det D^2u=1\quad\text{in }\mathbb R^n$$ 
must be a quadratic polynomial. Caffarelli \cite{Caf} extended this result to viscosity solutions (the proof can be also found in \cite[Theorem 1.1]{CL}). For another direction of extension, Jin and Xiong \cite{JX} studied the class of equations 
\begin{equation}\label{eq:MA-d}
\det D^2u(x,y)=|y|^{\alpha}
\end{equation}
 on the whole plane
$\mathbb R^2$, and established a Liouville theorem.
\begin{thm}[{\cite[Theorem 1.1]{JX}}]\label{thm:MA-d}
Let $u(x,y)$ be convex generalized (or Alexandrov) solution to \eqref{eq:MA-d}
with $\alpha>-1$. Then there exist some constants $A>0$, $B\in\mathbb R$ and a linear function $l(x,y)$ such that
\begin{equation}\label{eq:s-MA-d}
u(x,y)=\frac{1}{2A}x^2+\frac{AB^2}{2}y^2+Bxy+\frac{A}{(2+\alpha)(1+\alpha)}|y|^{2+\alpha}+l(x,y).
\end{equation}
\end{thm} 

At the end of this paper, we use the  approach above to give a new proof of this result in the case of $\alpha\geq 0$.  The main idea of Jin and Xiong in \cite{JX} is that using the partial Legendre transform to change \eqref{eq:MA-d} into a class of linearized Monge-Amp\`ere equations, then applying the Harnack inequality for linearized Monge-Amp\`ere equations derived by Caffarelli and Guti\'errez \cite{CG} and the scaling argument to classify all solutions of the transformed equation.  Our new proof is similar to Theorem \ref{thm:MA-d-hf} and Theorem \ref{thm:Liou-d-hf-ne}. 



The structure of this paper is as follows. In Section \ref{linear-divergence}, we derive the Liouville theorems for a class of linear elliptic equations in divergence form including \eqref{linear-div}. Then we prove Theorem \ref{thm:MA-d-hf}, Theorem \ref{thm:Liou-d-hf-ne} and  Theorem \ref{thm:MA-d} in Section \ref{sect:pf-thm-2}.

\vskip 20pt

\section{Liouville theorems for linear elliptic equations in divergence form}\label{linear-divergence}

In this section, we establish a Liouville theorem for a class of linear elliptic equations in divergence form, which may be degenerate or singular cases, in the half space. This theorem can be viewed as an extension of {\cite[Theorem 1.1]{WZ}}. The proof is very similar to \cite[Theorem 1.1]{WZ}, where the method of moving sphere will be used. Denote $\mathbb{R}_l^{n}=\{x=(x',x_n):x'\in\mathbb{R}^{n-1},\,x_n>l\}$ for $l\geq 0$.
\begin{thm}\label{thm:liou-div-hf}
For $n \geq 2$ and $a \in \mathbb{R}$, let $u\in C^2(\mathbb{R}_l^{n}) \cap C^0(\overline{\mathbb{R}_l^{n}})$ be a solution to
\[
\begin{cases}\operatorname{div}\left(x_n^a \nabla u\right)=0, & u>-C_0 \text { in }\mathbb{R}_l^{n}, \\[4pt]
u(x',l)=0, & \text { on }   \mathbb{R}^{n-1}\times\{x_n=l\},\end{cases}
\]
where $l\geq 0$ and $C_0>0$.
Then $u=C_* \left(x_n^{1-a}-l^{1-a}\right)$ for some nonnegative constant $C_*$. In particular,  when  $a\geq 1$, $C_*=0$.
\end{thm}
\begin{rem}
When $l=0$, Theorem \ref{thm:liou-div-hf} is just the Theorem 1.1 of \cite{WZ}.
\end{rem}
\begin{proof}[Proof of Theorem \ref{thm:liou-div-hf}]
We  extend $u$ to $\overline{\mathbb{R}^{n}_+}$ by letting $u\left(x', x_{n}\right)=0$ in $\mathbb{R}^{n-1} \times[0,l)$, and denote it by $\widetilde{u}$. Hence, we know that $\widetilde{u}(x)\in C\left(\overline{\mathbb{R}^{n}_+}\right)$.

Firstly, we show that $\widetilde{u}$ is weakly differentiable in $\mathbb{R}_+^{n}$ and
$$\nabla\widetilde{u}=\left\{\begin{aligned}&\nabla u,&& \mathbb{R}_l^{n},\\&0,&&\mathbb{R}^{n-1} \times(0,l).\end{aligned}\right.$$
Indeed, $\forall\,\varphi\in C_0^\infty(\mathbb{R}_+^{n})$, by integration by parts, we have
\[\int_{\mathbb{R}_+^{n}}\widetilde{u}\,\partial_{x_i}\varphi\mathrm{~d}x=\int_{\mathbb{R}^{n-1} \times(l,+\infty)}{u}\,\partial_{x_i}\varphi\mathrm{~d}x
=-\int_{\mathbb{R}^{n-1} \times(l,+\infty)}\partial_{x_i}{u}\,\varphi\mathrm{~d}x\]
for $i\leq n-1$ and
\begin{align*}
\int_{\mathbb{R}_+^{n}}\widetilde{u}\,\partial_{x_n}\varphi\mathrm{~d}x&=\int_{\mathbb{R}^{n-1} \times(l,+\infty)}{u}\,\partial_{x_n}\varphi\mathrm{~d}x\\
&=-\int_{\mathbb{R}^{n-1} \times\{x_n=l\}}{u}\,\varphi\mathrm{~d}{x'}-\int_{\mathbb{R}^{n-1} \times(l,+\infty)}\partial_{x_n}{u}\,\varphi\mathrm{~d}x\\
&=-\int_{\mathbb{R}^{n-1} \times(l,+\infty)}\partial_{x_n}{u}\,\varphi\mathrm{~d}x,
\end{align*}
where we used ${u}(x',l)=0$ in the last equality. 

Next, we show that $\widetilde{u}\in W^{1,2}_{loc}\left(\mathbb{R}_+^{n}\right)\cap C\left(\overline{\mathbb{R}^{n}_+}\right)$ is a weak solution to
\begin{equation}\label{eq:w-div}
\begin{cases}\operatorname{div}\left(x_n^a \nabla \widetilde{u}\right)=0, & \widetilde{u}>-C_0 \text { in } \mathbb{R}^{n}_+, \\[4pt]
\widetilde{u}=0, & \text { on }   \partial\mathbb{R}^{n}_+.\end{cases}
\end{equation}
Indeed, for any $\varphi\in C_0^\infty(\mathbb{R}_+^{n})$, there is
\begin{align*}
\int_{\mathbb{R}_+^{n}}\operatorname{div}\left(x_n^a \nabla \widetilde{u}\right)\varphi\mathrm{~d}x&=\int_{\mathbb{R}^{n-1} \times(l,+\infty)}\operatorname{div}\left(x_n^a \nabla \widetilde{u}\right)\varphi\mathrm{~d}x+\int_{\mathbb{R}^{n-1} \times(0,l)}\operatorname{div}\left(x_n^a \nabla \widetilde{u}\right)\varphi\mathrm{~d}x\\
&=-\int_{\mathbb{R}^{n-1} \times(l,+\infty)}x_n^a \nabla {u}\cdot\nabla\varphi\mathrm{~d}x-\int_{\mathbb{R}^{n-1} \times\{x_n=l\}}\partial_{x_n}{u}\,\varphi\mathrm{~d}{x'}\\
&=\int_{\mathbb{R}^{n-1} \times(l,+\infty)}\operatorname{div}\left(x_n^a \nabla {u}\right)\varphi\mathrm{~d}x=0.
\end{align*}
It's clear that $\widetilde{u}>-C_0$ in $\mathbb{R}^{n}_+$ and $\widetilde{u}=0$ on $\partial\mathbb{R}^{n}_+$. Hence, $\widetilde{u}\in W^{1,2}_{loc}\left(\mathbb{R}_+^{n}\right)\cap C\left(\overline{\mathbb{R}^{n}_+}\right)$ is a weak solution to \eqref{eq:w-div}. 

For any fixed $x \in \partial \mathbb{R}_{+}^{n}$ and $\lambda>0$, by
 Kelvin transformation
\[
y^{x, \lambda}=x+\frac{\lambda^{2}(y-x)}{|y-x|^{2}}, \quad \forall y \in \overline{\mathbb{R}_{+}^{n}},
\]
we define
\[
\widetilde{u}_{x, \lambda}(y)=\frac{\lambda^{n-2+a}}{|y-x|^{n-2+a}} \widetilde{u}\left(y^{x, \lambda}\right), \quad \forall y \in \overline{\mathbb{R}_{+}^{n}} .
\]
By \cite[Theorem 2.1]{YD}, we know that $\widetilde{u}_{x, \lambda}(y)\in W^{1,2}_{loc}\left(\mathbb{R}_+^{n}\right)$ satisfies
$\operatorname{div}\left(y_n^a \nabla \widetilde{u}_{x, \lambda}\right)=0$ in the weak sense, i.e. $\widetilde{u}_{x, \lambda}$ satisfies the same equation. 


For $a>2-n$, we consider $\bar u=\widetilde{u}+C_0$ instead of $\widetilde{u}$. Then $\displaystyle\lim _{|y| \rightarrow 0} \bar u(x+y)=C_0$ for $x \in \partial \mathbb{R}_{+}^{n}$.
Let 
\[
w_{x, \lambda}(y)=\bar u(y)-\bar u_{x, \lambda}(y), \quad \forall\, y \in \mathbb{R}_{+}^{n} .
\]
We have
\begin{equation*}
\varliminf_{|y| \rightarrow+\infty} w_{x, \lambda}(y) \geq 0-\lim_{|y| \rightarrow+\infty} \frac{\lambda^{n-2+a}}{|y-x|^{n-2+a}} \bar u\left(x+\frac{\lambda^{2}(y-x)}{|y-x|^{2}}\right)  = 0.
\end{equation*}
By the maximum principle, we have
$\widetilde{u}_{x, \lambda}(y)\leq\widetilde{u}(y),\,\,\forall\,y\in\mathbb R^n_+\backslash B_\lambda(x)$.
Hence by Lemma \ref{lem:moving-sphere-hf} below, we know that $\widetilde{u}(y',y_n)=\widetilde{u}(y_n)$. Then solving the corresponding ODE gives us the desired result.

For $a<2-n$, we consider $\bar u=\widetilde{u}-1$ instead of $\widetilde{u}$. Then $\displaystyle\lim _{|y| \rightarrow 0} \bar u(x+y)=-1$ for $x \in \partial \mathbb{R}_{+}^{n}$.
Let 
\[
w_{x, \lambda}(y)=\bar u(y)-\bar u_{x, \lambda}(y), \quad \forall\, y \in \mathbb{R}_{+}^{n} .
\]
We have
\begin{equation*}
\begin{aligned}
\varliminf_{|y| \rightarrow+\infty} w_{x, \lambda}(y) &=\varliminf_{|y| \rightarrow+\infty} \bar u(y)-\lim_{|y| \rightarrow+\infty} \frac{|y-x|^{2}}{\lambda^{2}} \bar u\left(x+\frac{\lambda^{2}(y-x)}{|y-x|^{2}}\right) \\
& \geq-1-C_0+\lim_{|y| \rightarrow+\infty} \frac{|y-x|^{2}}{\lambda^{2}} \\
&=+\infty.
\end{aligned}
\end{equation*}
Again by the maximum principle, we have
$\widetilde{u}_{x, \lambda}(y)\leq\widetilde{u}(y),\,\,\forall\,y\in\mathbb R^n_+\backslash B_\lambda(x)$.
Similarly, by Lemma \ref{lem:moving-sphere-hf}, we also have $\widetilde{u}(y',y_n)=\widetilde{u}(y_n)$, then we can obtain the conclusion.

As for $a=2-n$, we need to modify $\widetilde{u}_{x, \lambda}(y)$ to be
$$\widetilde{u}_{x, \lambda}(y)=\widetilde{u}\left(y^{x, \lambda}\right)+\ln\frac{\lambda}{|y-x|}.$$
Then  by similar arguments, we also have
$\widetilde{u}_{x, \lambda}(y)\leq\widetilde{u}(y),\,\,\forall\,y\in\mathbb R^n_+\backslash B_\lambda(x)$.
The result follows by applying Lemma \ref{lem:ms-ln}.
\end{proof}

In the proof of Theorem \ref{thm:liou-div-hf}, we used two crucial lemmas of moving spheres \cite{WZ}. For readers' convenience, we include a proof here, which is very similar to the proof of \cite[Lemma 5.7]{Li}.

\begin{lem}[{\cite[Lemma 3.3]{WZ}}]\label{lem:moving-sphere-hf}
 Assume $f(y) \in C^{0}\left(\overline{\mathbb{R}_{+}^{n}}\right)$, $n \geq 2$, and $\tau \in \mathbb{R}$. Suppose
\begin{equation}\label{eq:m-s-in-hf}
\left(\frac{\lambda}{|y-x|}\right)^{\tau} f\left(x+\frac{\lambda^{2}(y-x)}{|y-x|^{2}}\right) \leq f(y)
\end{equation}
for $\lambda>0$, $x \in \partial \mathbb{R}_{+}^{n}$,   $y \in \mathbb{R}_{+}^{n}$ satisfying $|y-x| \geq \lambda$.
Then
\[
f(y)=f\left(y^{\prime}, y_{n}\right)=f\left(0^{\prime}, y_{n}\right), \quad \forall y=\left(y^{\prime}, y_{n}\right) \in \mathbb{R}_{+}^{n} .
\]
\end{lem}
\begin{proof}
For any fixed $y',z'\in\mathbb R^{n-1}$ with $y'\neq z'$ and $y_n>0$, we
denote $y=(y',y_n)$ and $z=(z',z_n)$, where $z_n=\frac{b-1}{b}y_n$ for $b>1$.
Then we have
$$x=y+b(z-y)\in\partial \mathbb{R}_{+}^{n}$$
and
$$z=x+\frac{\lambda^2(y-x)}{|y-x|^2},$$
where $\lambda=\sqrt{|z-x|\cdot |y-x|}$. By \eqref{eq:m-s-in-hf}, we have
\begin{equation}\label{eq:4.1-z}
\left(\frac{\lambda}{|y-x|}\right)^{\tau} f(z) \leq f(y).
\end{equation}
Since
$$\displaystyle\lim_{b\to+\infty}\frac{\lambda}{|y-x|}=\lim_{|x|\to\infty}\sqrt{\frac{|z-x|}{|y-x|}}=1,\quad\lim_{b\to+\infty}z_n=\lim_{b\to+\infty}\frac{b-1}{b}y_n=y_n.$$
and $f$ is continuous, we have
$f(z',y_n)\leq f(y',y_n)$.
By the arbitrariness of $y'\neq z'$, the proof is completed.
\end{proof}

\begin{lem}\label{lem:ms-ln}
Suppose that $f \in C^{0}\left(\overline{\mathbb{R}_{+}^{n}}\right)$ satisfies that for all $x \in \partial \mathbb{R}_{+}^n$ and $\lambda>0$,
\begin{equation*}
f(y) \geq f\left(x+\frac{\lambda^2(y-x)}{|y-x|^2}\right)+\ln \frac{\lambda}{|y-x|}, \quad \forall y \in \mathbb{R}_{+}^n \backslash B_\lambda(x) .
\end{equation*}
Then
\[
f(y)=f\left(y^{\prime}, y_n\right)=f\left(0^{\prime}, y_n\right), \quad \forall y=\left(y^{\prime}, y_n\right) \in \mathbb{R}_{+}^n.
\]
\end{lem}
\begin{proof}
The proof is the same as Lemma \ref{lem:moving-sphere-hf}. It suffices to replace \eqref{eq:4.1-z} by
$$\ln\frac{\lambda}{|y-x|}+ f(z) \leq f(y).$$
\end{proof}

A Liouville theorem for the Neumman boundary value is also derived in \cite{WZ}.
\begin{thm}[{\cite[Theorem 1.2]{WZ}}]\label{thm:liou-div-hf-ne}
Assume $n \geq 2$ and $\max \{-1,2-n\}<a<1$. Suppose $u(x) \in$ $C^2\left(\mathbb{R}_{+}^n\right) \cap C^1\left(\overline{\mathbb{R}_{+}^n}\right)$ satisfies
\begin{equation}\label{eq:MA-ne}
\begin{cases}\operatorname{div}\left(x_n^a \nabla u\right)=0, & u>0, \quad \text { in } \mathbb{R}_{+}^n, \\[3pt]
 x_n^a \frac{\partial u}{\partial x_n}=0 & \text { on } \partial \mathbb{R}_{+}^n .\end{cases}
\end{equation}
Then $u=C$ for some positive constant $C$.
The boundary condition in \eqref{eq:MA-ne} holds in the following sense:
\[
\displaystyle\lim _{x_n \rightarrow 0^{+}} x_n^a \frac{\partial u}{\partial x_n}=0 .
\]
\end{thm}

\vskip 20pt

\section{Proof of main theorems}\label{sect:pf-thm-2}

In this section, we first derive the new equation under the partial Legendre transform. 
Let $\Omega\subset \mathbb R^2$ and $u(x, y)$ be a uniformly convex function on $\Omega$.
The partial Legendre transform in the $x$-variable is
\begin{equation}\label{eq:part-legendre}
u^{\star}(\xi, \eta)=x u_{x}(x, y)-u(x, y),
\end{equation}
where
\begin{equation}\label{eq:p-Le-tr}
(\xi, \eta)=\mathcal{P}(x, y):=\left(u_{x}, y\right) \in \mathcal{P}(\Omega):=\Omega^{\star} .
\end{equation}
We have
$$
\frac{\partial(\xi, \eta)}{\partial(x, y)}=\left(\begin{array}{ccc}
u_{x x} & & u_{x y} \\
0 & & 1
\end{array}\right), \quad \text { and } \quad \frac{\partial(x, y)}{\partial(\xi, \eta)}=\left(\begin{array}{ccc}
\frac{1}{u_{x x}} & & -\frac{u_{x y}}{u_{x x}} \\
0 & & 1
\end{array}\right).
$$
Hence,
\begin{eqnarray}
&& u_{\xi}^{\star}=x,\ \ u_{\eta}^{\star}=-u_{y}, \label{eq:part-le-d}
\\
&& u_{\xi \xi}^{\star}=\frac{1}{u_{x x}},\ \  u_{\eta \eta}^{\star}=-\frac{\operatorname{det} D^{2} u}{u_{x x}},\ \ u_{\xi \eta}^{\star}=-\frac{u_{x y}}{u_{x x}}.
\end{eqnarray}
Then if $u$
is a solution to  
\[\det D^2u=(a+bx)^\alpha,\]
$u^\star$ is a solution to
\[(a+b\eta)^{\alpha}u_{\xi \xi}^{\star}+u_{\eta\eta}^{\star}=0.\]
We will apply the results in Section \ref{linear-divergence} related to this equation  to prove the main theorems.


\subsection{The case of Dirichlet boundary value}
We use Theorem \ref{thm:liou-div-hf} to prove Theorem \ref{thm:MA-d-hf}.
\begin{proof}[Proof of Theorem \ref{thm:MA-d-hf}]
We consider the the partial Legendre transform $u^\star$ of $u$  on $\mathbb R^2_+$.
Note that $\xi=u_x$, $\eta=y$ and $u_x=x$ on $\{y=0\}$ by \eqref{eq:MA-d-hf}, which gives us that $\mathcal{P}(\{y=0\})=\{\eta=0\}$. Hence, we have $\mathcal{P}\left(\mathbb{R}_{+}^{2}\right)=\mathbb{R}_{+}^{2}$. Then if $u$
is a solution to  \eqref{eq:MA-d-hf}, $u^\star$ is a solution to
\begin{equation}\label{eq:MA-d-hf-tr}
\left\{
\begin{aligned}
(a+b\eta)^{\alpha}u_{\xi \xi}^{\star}+u_{\eta\eta}^{\star}&=0&&\text{ in }\mathbb R\times (0,+\infty),\\
u^{\star}(\xi,0)&=\frac{\xi^2}{2}&&\text{ on }\mathbb R \times \{0\},
\end{aligned}
\right.
\end{equation}
where we used the fact that the Legendre transform of $x \mapsto \frac{1}{2} x^{2}$ is $\xi \mapsto \frac{1}{2} \xi^{2}$. 

Since Legendre transform does not change the convexity, we have that $u^\star_{\xi\xi}\geq 0$. 
Denote $v:=u^\star_{\xi\xi}-1$.
Differentiating \eqref{eq:MA-d-hf-tr} twice respect to $\xi$, we have that $v\geq -1$ solves
\begin{equation}\label{eq:MA-d-hf-tr-v}
\left\{
\begin{aligned}
(a+b\eta)^{\alpha}v_{\xi \xi}+v_{\eta\eta}&=0&&\text{ in }\mathbb R\times (0,+\infty),\\[3pt]
v(\xi,0)&=0&&\text{ on }\mathbb R\times\{\eta=0\}.
\end{aligned}
\right.
\end{equation}
Let $\xi=x_1$, $\eta=f(x_2)=b^{\frac{-\alpha}{\alpha+2}}\left(\frac{\alpha+2}{2}x_2\right)^{\frac{2}{\alpha+2}}-\frac{a}{b}$ and 
\begin{equation*}
\widetilde{v}(x_1,x_2)=v\left(x_1,f(x_2)\right).
\end{equation*}
A direct calculation yields
\begin{align*}
&\widetilde{v}_{11}=v_{\xi\xi},\\
&\widetilde{v}_{2}=b^{\frac{-\alpha}{\alpha+2}}\left(\frac{\alpha+2}{2}x_2\right)^{\frac{-\alpha}{\alpha+2}}v_\eta,\\
&\widetilde{v}_{22}=-\frac{\alpha}{\alpha+2}x_2^{-1}\widetilde{v}_{2}+(a+b\eta)^{-\alpha}v_{\eta\eta}.
\end{align*}
$\eta=0$ gives us that $x_2=\frac{2}{b(\alpha+2)}a^{\frac{\alpha+2}{2}}$. Denote $l=\frac{2}{b(\alpha+2)}a^{\frac{\alpha+2}{2}}$.
Hence by \eqref{eq:MA-d-hf-tr-v}, we know that $\tilde v\geq -1$ solves
\begin{equation*}
\left\{
\begin{aligned}
\widetilde{v}_{11}+\widetilde{v}_{22}+\frac{\alpha}{\alpha+2}x_2^{-1}\widetilde{v}_{2}&=0&&\text{ in }\mathbb R\times (l,+\infty),\\[3pt]
\widetilde{v}(x_1,0)&=0&&\text{ on }\mathbb R\times\{x_2=l\},
\end{aligned}
\right.
\end{equation*}
i.e.,
\begin{equation*}
\left\{
\begin{aligned}
\operatorname{div}\left(x_2^{\frac{\alpha}{\alpha+2}}\nabla \widetilde{v}\right)&=0&&\text{ in }\mathbb R\times (l,+\infty),\\[3pt]
\widetilde{v}(x_1,0)&=0&&\text{ on }\mathbb R\times\{x_2=l\}.
\end{aligned}
\right.
\end{equation*}
Applying Theorem \ref{thm:liou-div-hf} with $n=2$ and $a=\frac{\alpha}{\alpha+2}<1$, we know that $\widetilde{v}(x_1,x_2)=C_*\left({x_2}^{\frac{2}{\alpha+2}}-l^{\frac{2}{\alpha+2}}\right)$ for some nonnegative constant $C_*$. Transforming back to $(\xi,\eta)$, we have $v(\xi,\eta)=A\eta$ for some $A\geq 0$, i.e.,
$u_{\xi \xi}^{\star}(\xi, \eta)=1+A \eta$.
Then
\[
u^{\star}(\xi, \eta)=h_{1}(\eta)+\xi h_{2}(\eta)+\frac{\xi^{2}}{2}(1+A \eta)
\]
for some functions $h_{1}, h_{2}:[0,+\infty) \rightarrow \mathbb{R}$. Recalling \eqref{eq:MA-d-hf-tr}, we have $h_{1}(0)=h_{2}(0)=0$ and
\[
 h_{1}^{\prime \prime}(\eta)+\xi h_{2}^{\prime \prime}(\eta)+(1+A \eta)(a+b\eta)^{\alpha}=0\]
 on $\mathbb{R} \times(0,+\infty)$.
This implies that $h_{1}^{\prime \prime}(\eta)+(1+A \eta)(a+b\eta)^{\alpha}=0$ and $h_{2}^{\prime \prime}(\eta)=0$. By solving the ODEs, we obtain
\begin{equation*}
u^\star(\xi,\eta)=\left\{
\begin{aligned}
&B\eta-\frac{(b-aA)(a+b\eta)^{2+\alpha}}{b^3(1+\alpha)(2+\alpha)}-\frac{A(a+b\eta)^{3+\alpha}}{b^3(2+\alpha)(3+\alpha)}+C\xi\eta\\
&\quad+\frac{(b-aA)a^{2+\alpha}}{b^3(1+\alpha)(2+\alpha)}+\frac{Aa^{3+\alpha}}{b^3(2+\alpha)(3+\alpha)}+\frac{\xi^{2}}{2}(1+A \eta),&&\alpha\neq -1;\\[8pt]
&B\eta-\frac{b-aA}{b^3}(a+b\eta)\ln(a+b\eta)-\frac{A}{2}\eta^2+C\xi\eta\\
&\quad\quad\quad\quad\quad\quad\quad\quad\quad\,\,+\frac{(b-aA)a\ln a}{b^3}+\frac{\xi^{2}}{2}(1+A \eta),&&\alpha= -1,
\end{aligned}
\right.
\end{equation*}
for some constants $B, C \in \mathbb{R}$. Recalling that the Legendre transform is an involution on convex functions, we recover $u$ by taking the partial Legendre transform of $u^{\star}$ :
\begin{equation*}
u(x,y)=\left\{
\begin{aligned}
&\frac{(b-aA)(a+by)^{2+\alpha}}{b^3(1+\alpha)(2+\alpha)}+\frac{A(a+by)^{3+\alpha}}{b^3(2+\alpha)(3+\alpha)}-By\\
&\quad-\frac{(b-aA)a^{2+\alpha}}{b^3(1+\alpha)(2+\alpha)}-\frac{Aa^{3+\alpha}}{b^3(2+\alpha)(3+\alpha)}+\frac{(x-Cy)^2}{2(1+Ay)},&&\alpha\neq -1;\\[8pt]
&\frac{b-aA}{b^3}(a+by)\ln(a+by)+\frac{A}{2b}y^2-By\\
&\quad\quad\quad\quad\quad\quad\quad\quad\quad\,\,-\frac{(b-aA)a\ln a}{b^3}+\frac{(x-Cy)^2}{2(1+Ay)},&&\alpha= -1.
\end{aligned}
\right.
\end{equation*}
This gives us a complete classification of all solutions to \eqref{eq:MA-d-hf}.
\end{proof}

\begin{rem}\label{rem:sharp}
$\alpha>-2$  in Theorem \ref{thm:MA-d-hf} is sharp since \eqref{eq:MA-d-hf} has no convex solutions continuous up to boundary in $\mathbb R^2_+$ when $\alpha\leq -2$. Indeed, if there exists a convex function $u\in C^2({\mathbb R_+^2})\cap C(\overline{\mathbb R_+^2})$ solves \eqref{eq:MA-d-hf}, by \cite[Theorem 5.1]{S2}, we will have a Pogorelov type estimate
$$(1-u)u_{xx}\leq C(\max|u_x|)$$
in $S_1$, where $S_h=\{x\in\mathbb R^2_+:u(x)<u(0)+\nabla u(0)\cdot x+h\}$ for $h>0$. Since $u(x,0)=\frac{1}{2}x^2 \text{on }\partial\mathbb R^2_+$, we know that $|u_x|$ is bounded in $S_1$ (depends on $\|u\|_{L^\infty(S_2)}$). Then there exists a small $c_0>0$ such that $u_{xx}\leq C(\|u\|_{L^\infty(S_2)})$ in $B_{c_0}^+$. Hence, we have
\begin{align*}
Cu_{yy}\geq u_{xx}u_{yy}\geq u_{xx}u_{yy}-u_{xy}^2=y^\alpha \text{ in } B_{c_0}^+,
\end{align*}
i.e. in $B_{c_0}^+$. Then it holds
$$u(x,y)\geq \left\{\begin{aligned}
&\frac{1}{C(1+\alpha)(2+\alpha)}y^{2+\alpha}+D(x)y+E(x),&&\alpha<-2,\\[4pt]
&-\frac{1}{C}\ln y+D(x)y+E(x),&&\alpha=-2,
\end{aligned}\right.$$
which means that
$\displaystyle\lim_{y\to 0+}u(x,y)=+\infty$. This
contradicts with $u\in C(\overline{\mathbb R_+^2})$. 
\end{rem}

\vskip 10pt

\subsection{The  case of Neumann boundary value}\label{sect:pf-thm-3}

We prove Theorem \ref{thm:Liou-d-hf-ne} in this section. 

\begin{proof}[Proof of Theorem \ref{thm:Liou-d-hf-ne}]
For $\alpha=0$, it has been proved by Jian and Tu \cite[Theorem 1.1]{JT}. In the following, we mainly prove the case for $\alpha>0$ .

Similarly as in the last section, we know that  if $u$
is a solution to  \eqref{eq:MA-d-hf-ne}, $u^\star$ is a solution to
\begin{equation}\label{eq:MA-d-hf-tr-ne}
\left\{
\begin{aligned}
\eta^{\alpha}u_{\xi \xi}^{\star}+u_{\eta\eta}^{\star}&=0&&\text{ in }\mathbb R\times (0,+\infty),\\[4pt]
u^{\star}_\eta(\xi,0)&=0&&\text{ on }\mathbb R  \times \{0\}.
\end{aligned}
\right.
\end{equation}
Denote $v:=u^\star_{\xi\xi}$.
Differentiating \eqref{eq:MA-d-hf-tr-ne} twice respect to $\xi$, we have that $v> 0$ solves
\begin{equation}\label{eq:MA-d-hf-tr-v-ne}
\left\{
\begin{aligned}
\eta^{\alpha}v_{\xi \xi}+v_{\eta\eta}&=0&&\text{ in }\mathbb R\times (0,+\infty),\\
v_\eta(\xi,0)&=0&&\text{ on }\mathbb R\times\{\eta=0\}.
\end{aligned}
\right.
\end{equation}
Let $\xi=x_1$, $\eta=\left(\frac{\alpha+2}{2}\right)^{\frac{2}{\alpha+2}}x_2^{\frac{2}{\alpha+2}}$ and 
\[\widetilde{v}(x_1,x_2)=v\left(x_1,\left(\frac{\alpha+2}{2}\right)^{\frac{2}{\alpha+2}}x_2^{\frac{2}{\alpha+2}}\right).\]
Then \eqref{eq:MA-d-hf-tr-v-ne} gives us that $\tilde v> 0$ solves
\begin{equation*}
\left\{
\begin{aligned}
\widetilde{v}_{11}+\widetilde{v}_{22}+\frac{\alpha}{\alpha+2}x_2^{-1}\widetilde{v}_{2}&=0&&\text{ in }\mathbb R\times (0,+\infty),\\
x_2^{\frac{\alpha}{\alpha+2}}\widetilde{v}_2(x_1,0)&=0&&\text{ on }\mathbb R\times\{x_2=0\},
\end{aligned}
\right.
\end{equation*}
i.e.,
\begin{equation*}
\left\{
\begin{aligned}
\operatorname{div}\left(x_2^{\frac{\alpha}{\alpha+2}}\nabla \widetilde{v}\right)&=0&&\text{ in }\mathbb R\times (0,+\infty),\\
x_2^{\frac{\alpha}{\alpha+2}}\widetilde{v}_2(x_1,0)&=0&&\text{ on }\mathbb R\times\{x_2=0\}.
\end{aligned}
\right.
\end{equation*}
Applying Theorem \ref{thm:liou-div-hf-ne} with $n=2$ and $a=\frac{\alpha}{\alpha+2}\in(0,1)$, we know that $\widetilde{v}=C$ for some positive constant $C$. Transforming back to $(\xi,\eta)$, we have $v(\xi,\eta)=A$ for some $A> 0$, i.e., $u_{\xi \xi}^{\star}(\xi, \eta)=A$
 for some $A > 0$.
Then
\[
u^{\star}(\xi, \eta)=h_{1}(\eta)+\xi h_{2}(\eta)+\frac{A}{2}\xi^{2}
\]
for some functions $h_{1}, h_{2}:[0,+\infty) \rightarrow \mathbb{R}$. Recalling \eqref{eq:MA-d-hf-tr-ne}, we have
\[
h'_{1}(0)=h'_{2}(0)=0 \quad \text {and} \quad h_{1}^{\prime \prime}(\eta)+\xi h_{2}^{\prime \prime}(\eta)+A\eta^{\alpha}=0 \quad \text {on } \mathbb{R} \times(0,+\infty).
\]
This implies that $h_{1}^{\prime \prime}(\eta)+A\eta^{\alpha}=0$ and $h_{2}^{\prime \prime}(\eta)=0$. By solving these ODEs, we obtain
\begin{equation*}
u^\star(\xi,\eta)=\frac{A}{2}\xi^{2}+B\xi-\frac{A\eta^{2+\alpha}}{(1+\alpha)(2+\alpha)}+C
\end{equation*}
for some constants $A, C \in \mathbb{R}$. Recalling that $u=(u^{\star})^\star$, we have
\begin{equation*}
u(x,y)=\frac{1}{2A}(x-B)^2+\frac{Ay^{2+\alpha}}{(1+\alpha)(2+\alpha)}-C,
\end{equation*}
which yields \eqref{eq:s-MA-d-hf-ne}.
\end{proof}

\vskip 10pt

\subsection{The entire space case}\label{sect:pf-thm-1}

Before proving Theorem \ref{thm:MA-d}, we first recall two theorems for \eqref{eq:MA-d} in \cite{JX}. 
\begin{thm}[{\cite[Theorem 4.1]{JX}}]\label{thm:thm-4.1}
Let $\Omega$ be an open convex set in $\mathbb{R}^2$, and $u$ be the generalized solution of
\[
\operatorname{det} D^2 u(x)=\left|y\right|^\alpha\quad\text {in } \Omega \text {, }
\]
with $u=0$ on $\partial \Omega$. Then $u$ is strictly convex in $\Omega$, and $u \in C_{\text {loc }}^{1, \delta}(\Omega)$ for some $\delta>0$ depending only on $\alpha$. Furthermore, the partial Legendre transform $u^\star$ of $u$ is a strong solution of
\[
|\eta|^{\alpha}u^\star_{\xi\xi}+u^\star_{\eta\eta}=0\quad\text{in }\mathcal{P}(\Omega),
\]
where the map $\mathcal{P}$ is given in \eqref{eq:p-Le-tr}.
\end{thm}

\begin{thm}[{\cite[Theorem 4.2]{JX}}]\label{thm:thm-4.2}
Let $u$ be a generalized solution of
\[
\operatorname{det} D^2 u(x)=\left|y\right|^\alpha\quad\text {in } \mathbb R^2 .
\] 
Then $u$ is strictly convex.
\end{thm}

Hence Theorem \ref{thm:thm-4.1} and Theorem \ref{thm:thm-4.2} give us that $u$ is strictly convex, and then $u$ is smooth away from $\left\{y=0\right\}$. Furthermore, we know that $u \in C_{\text {loc }}^{1, \delta}\left(\mathbb{R}^2\right)$ and the partial Legendre transform $u^\star$ of $u$ is a strong solution of \eqref{eq:MA-d-tr}.

Next, we need a Liouville theorem for degenerate elliptic equations in divergence form. This theorem is a partial extension of \cite[Corollary 1.4]{WZ}, where they assumed stronger conditions.
\begin{thm}\label{thm:liou-div}
Assume that $n = 2$ and $a\geq 0$. Then any positive $C^1(\mathbb R^2)$ solution to
\begin{equation}\label{eq:ind-a}
\operatorname{div}\left(|x_2|^a\nabla u\right)=0\quad\text{in }\mathbb R^2
\end{equation}
is a constant function.
\end{thm}
\begin{proof}
Note that $u$ only belongs to $C^1(\mathbb R^2)$. We need to prove this theorem in the weak sense.
Hence we can repeat the same process as in the proof of Theorem \ref{thm:liou-div-hf}. Due to its similarity, we omit the details here.
\end{proof}

Now, we are ready to give the proof of Theorem \ref{thm:MA-d}.

\begin{proof}[Proof of Theorem \ref{thm:MA-d}] Our proof only works for the case $\alpha\geq 0$.
We consider the partial Legendre transform $u^\star$ of $u$. $u^\star$
is a solution to 
\begin{equation}\label{eq:MA-d-tr}
|\eta|^{\alpha}u^\star_{\xi\xi}+u^\star_{\eta\eta}=0\quad\text{in }\mathbb R^2.
\end{equation}
Let $v:=u^\star_{\xi\xi}\geq 0$. Differentiating \eqref{eq:MA-d-tr} twice respect to $\xi$, we have that $v\geq 0$ solves
\begin{equation}\label{eq:MA-d-tr-v}
|\eta|^{\alpha}v_{\xi\xi}+v_{\eta\eta}=0\quad\text{in }\mathbb R^2.
\end{equation}
By a change of variables, we let
$$\widetilde{v}(x_1,x_2)=\left\{\begin{aligned}
&v\left(x_1,\left(\frac{\alpha+2}{2}\right)^{\frac{2}{\alpha+2}}x_2^{\frac{2}{\alpha+2}}\right),&&\eta\geq 0,\\[5pt]
&v\left(x_1,-\left(\frac{\alpha+2}{2}\right)^{\frac{2}{\alpha+2}}(-x_2)^{\frac{2}{\alpha+2}}\right),&&\eta<0.
\end{aligned}
\right.$$
A direct calculation yields
\begin{align*}
\widetilde{v}_{11}&=v_{\xi\xi},\\
\widetilde{v}_{22}&=-\frac{\alpha}{\alpha+2}x_2^{-1}\widetilde{v}_{2}+|\eta|^{-\alpha}v_{\eta\eta}.
\end{align*}
By \eqref{eq:MA-d-tr-v}, we know that $\widetilde{v}\geq 0$ solves
$$\widetilde{v}_{11}+\widetilde{v}_{22}+\frac{\alpha}{\alpha+2}x_2^{-1}\widetilde{v}_{2}=0\quad\text{in }\mathbb R^2,$$
i.e.,
$$\operatorname{div}\left(|x_2|^{\frac{\alpha}{\alpha+2}}\nabla \widetilde{v}\right)=0\quad\text{in }\mathbb R^2.$$
Hence, by Theorem \ref{thm:liou-div} with $a=\frac{\alpha}{\alpha+2}\geq 0$ in \eqref{eq:ind-a}, we obtain $\widetilde{v}\equiv$ constant. Thus $u^\star_{\xi\xi}\equiv A$, where $A$ is a constant. Similar to the proofs of Theorem \ref{thm:MA-d-hf} and Theorem \ref{thm:Liou-d-hf-ne}, by solving these ODEs, we have
$$u^\star(\xi,\eta)=\frac{A}{2}\xi^2-\frac{A}{(1+\alpha)(2+\alpha)}|\eta|^{2+\alpha}+B\xi\eta+l(\xi,\eta).$$
Again by $u=(u^\star)^\star$, we have \eqref{eq:s-MA-d}.
\end{proof}


\begin{thebibliography}{99999}

\bibitem[Ab]{Ab}
Abreu, M.,
K\"ahler geometry of toric varieties and extremal metrics.
\textit{Int. J. Math.} \textbf{9} (1998), no. 6, 641-651.

\bibitem[Ca]{Ca}
Calabi, E., 
Improper affine hyperspheres of convex type and a generalization of a theorem by K. J\"orgens.
\textit{Mich. Math. J.} \textbf{5} (1958), 105-126.

\bibitem[Caf]{Caf}
Caffarelli, L. A., 
\textit{Topics in PDEs: The Monge-Amp\`ere Equation, Graduate course}. 
Courant Institute, New York University, 1995.

\bibitem[CG]{CG}
Caffarelli, L. A., Guti\'errez, C. E., 
Properties of solutions of the linearized Monge-Amp\`ere equation.
\textit{Amer. J. Math.} \textbf{119} (1997), no. 2, 423-465.

\bibitem[Ch]{Ch}
Chern, S. S.:
Affine minimal hypersurfaces.
\textit{Minimal submanifolds and geodesics} (Proc. Japan-United States Sem., Tokyo, 1977), pp. 17-30, North-Holland, Amsterdam-New York, 1979.

\bibitem[CL]{CL}
Caffarelli, L. A., Li, Y. Y., 
An extension to a theorem of J\"orgens, Calabi, and Pogorelov. 
\textit{Comm. Pure Appl. Math.} \textbf{56} (2003), no. 5, 549-583.


\bibitem[CS]{CS}
Caffarelli, L., Silvestre, L., 
An extension problem related to the fractional Laplacian. 
\textit{Comm. Part. Diff. Eqns.} \textbf{32} (2007), 1245-1260.



\bibitem[Fi]{Fi}
Figalli, A., 
\textit{The Monge-Amp\`ere equation and its applications}. 
Zurich Lectures in Advanced Mathematics, European Mathematical Society (EMS), Z\"urich, 2017.




\bibitem[JL]{JL}
Jia, F., Li, A. M.,
A Bernstein property of some fourth order partial differential equations.
\textit{Results Math.} \textbf{56} (2009), 109-139.

\bibitem[JT]{JT}
 Jian, H. Y., Tu, X. S., 
A Liouville theorem for the Neumann problem of Monge-Amp\`ere equations. 
\textit{J. Funct. Anal.} \textbf{284} (2023), no. 6, Paper No. 109817.

\bibitem[JX]{JX}
Jin, T. L., Xiong, J. G., 
A Liouville theorem for solutions of degenerate Monge-Amp\`ere equations.
\textit{Comm. Part. Diff. Eqns.} \textbf{39} (2014), no. 2, 306-320.

\bibitem[J\"o]{Jo}
J\"orgens, K., \"Uber die L\"osungen der Differentialgleichung $rt - s^2 = 1$.
\textit{Math. Ann.} \textbf{127} (1954), 130-134.

\bibitem[Li]{Li}
Li, Y. Y.,
Remark on some conformally invariant integral equations: the method of moving spheres.
\textit{J. Eur. Math. Soc. (JEMS)} \textbf{6} (2004), no. 2, 153-180.


\bibitem[Po]{Po}
Pogorelov, A.V., 
On the improper affine hyperspheres.
\textit{Geom. Dedic.} \textbf{1} (1972), no. 1, 33-46.

\bibitem[Ru]{Ru}
Rubin, D.,
The Monge-Amp\`ere equation with Guillemin boundary conditions.
\textit{Cal. Var. Part. Diff. Eqns}. \textbf{54} (2015), 951-965.

\bibitem[S1]{S1}
Savin, O., 
Pointwise $C^{2,\alpha}$ estimates at the boundary for the Monge-Amp\`ere equation. 
\textit{J. Amer. Math. Soc.} \textbf{26} (2013), no. 1, 63-99.

\bibitem[S2]{S2}
Savin, O., 
A localization theorem and boundary regularity for a class of degenerate Monge-Amp\`ere equations.
\textit{J. Diff. Eqns.} \textbf{256} (2014), no. 2, 327-388.

\bibitem[SZ]{SZ}
Savin, O., Zhang, Q., 
Boundary regularity for Monge-Ampère equations with unbounded right hand side. 
\textit{Ann. Sc. Norm. Super. Pisa Cl. Sci. (5)} \textbf{20} (2020), no. 4, 1581-1619.

\bibitem[TW1]{TW1}
Trudinger, N. S., Wang, X. -J.,
The Bernstein problem for affine maximal hypersurfaces.
\textit{Invent. Math.} \textbf{140} (2000), no. 2, 399-422.

\bibitem [TW2] {TW2}
Trudinger, N. S., Wang, X. -J.,
Bernstein-J\"orgens theorem for a fourth order partial differential equation.
\textit{J. Part. Diff. Eqns.} {\textbf 15} (2002), no. 1, 78-88.



\bibitem[WZ]{WZ}
Wang, L., Zhu, M. J., 
Liouville theorems on the upper half space. 
\textit{Disc. Cont. Dyn. Syst.} \textbf{40} (2020), no. 9, 5373-5381. 

\bibitem[YD]{YD}
Yao, J. G., Dou, J. B.,
Classification of positive solutions to a divergent equation on the upper half space. 
\textit{Acta Math. Sin. (Engl. Ser.)} \textbf{38} (2022), no. 3, 499-509.

\bibitem[Z]{Z}
Zhou, B.,
The Bernstein theorem for a class of fourth order equations.
\textit{Calc. Var. Part. Diff. Eqns.} \textbf{43} (2012), no. 1-2, 25-44.

\end{thebibliography}
\end{document}